\documentclass[11pt]{article}
\usepackage[margin=1.1in]{geometry}
\usepackage{amsmath,amssymb,amsthm,amsfonts}
\usepackage[T1]{fontenc}
\usepackage{graphicx}
\graphicspath{{figures/}{./}}
\usepackage{booktabs}
\usepackage{xcolor}
\usepackage[colorlinks=true, linkcolor=blue!60!black,
            citecolor=blue!60!black, urlcolor=blue!60!black]{hyperref}
\newif\ifcomments
\commentstrue

\newtheorem{assumption}{Assumption}
\newtheorem{theorem}{Theorem}
\newtheorem{lemma}{Lemma}
\newtheorem{corollary}{Corollary}
\newtheorem{proposition}{Proposition}
\newtheorem{definition}{Definition}

\newenvironment{IEEEproof}[1][Proof]%
  {\par\noindent\textit{#1:}\ }{\hfill$\blacksquare$\par}
\title{Hidden Star-Convexity in Policy Optimization\\
for Gain-Scheduled LQR: Extended Version\thanks{This is the extended
version of the IEEE Control Systems Letters
submission~\cite{letter2026}: all proofs omitted there for space
appear here, together with extended experiments.}}
\author{Shiva Shakeri \and P\'eter Baranyi \and Mehran Mesbahi}
\date{}
\begin{document}
\maketitle
\begin{abstract}
This document is the extended version of the letter ``Hidden
Star-Convexity in Policy Optimization for Gain-Scheduled LQR.'' It
contains the complete proofs of all results stated there and
additional experiments. Section and theorem numbers in the main body mirror the
letter; material exclusive to this document appears in the appendices.
\end{abstract}
% =========================================================================
% MAIN BODY: mirrors the letter section-for-section. Once the letter's
% sections are frozen, paste them here (with \commentsfalse) and add
% pointers into the appendices where proofs were omitted.
% =========================================================================
\section{Introduction}
% ---- citation keys are placeholders until the .bib is built ----
% [A]  fazel2018global  (+ optional rate companion: bu2019lqr / mohammadi2021)
% [B0] takarics2014wing (JGCD: aeroelastic wing via relaxed TP framework)
% [B]  rugh2000survey + baranyi2004tp
% [C]  tanaka2001pdc (or TP-control construction paper)
Policy optimization for a single linear quadratic regulator is by now well
understood: although the cost is a nonconvex function of the feedback gain,
gradient descent converges to the global optimum at a linear
rate~\cite{fazel2018global}, with the first-order analysis
in~\cite{bu2019lqr}. This behavior has structural roots: under a classical change of variables built on the
closed-loop covariance, the single-plant problem is
convex~\cite{watanabe2025hidden}. Gain scheduling is the standard route to
operating more general designs across an envelope, from flight control at varying
airspeed to aeroelastic vibration suppression~\cite{takarics2014wing}. The
plant is modeled as a linear parameter-varying system, commonly in polytopic
or tensor-product form~\cite{rugh2000survey,baranyi2004tp}, and the
controller is a schedule of vertex gains interpolated through the same fixed
weighting functions that define the model~\cite{tanaka2001pdc}. Optimizing
this schedule by gradient descent couples the entire plant family through a
single set of decision variables, and the single-plant guarantee does not
carry over: the scheduled cost can admit strictly suboptimal local
minima, and the convexifying change of variables is obstructed by the
shared gains.
% [G] tanaka2001pdc (+ optional Scherer-school guaranteed-cost ref)
% [D] dr-lqr 2503.24371 + maml-lqr 2401.14534
% [E] sof landscape 2310.19022 (+ optional Bu topology)
% [H] giessler2025lpv (arXiv:2505.22248)
For polytopic LPV plants, scheduled gains are classically synthesized by
convex methods: common-Lyapunov linear matrix inequality conditions
guarantee stability and a bound on the cost, at a conservatism set by the chosen
vertex representation~\cite{tanaka2001pdc}. Optimizing the schedule
directly against the quadratic cost removes this conservatism at the price
of non-convexity, and the available convergence guarantees are partial. In
domain-randomized and meta-learning formulations, where the objective
likewise averages LQR costs over sampled plants, global convergence holds
only under explicit heterogeneity
bounds~\cite{drlqr2025,mamllqr2024}, or for a single lifted static
gain rather than a schedule~\cite{pctlqr2026}. The
landscape literature on structured feedback also examines the complementary
obstruction: shared parametrizations break gradient dominance and admit
spurious minima, and no convergence rates on quantified regions accompany
these results~\cite{soflandscape2023}. For linearly constrained gain
structures, policy updates over the constraint submanifold, under
Riemannian metrics adapted to the feedback geometry, recover
convergence guarantees~\cite{talebi2024submanifolds}. Gradient flows have also been
employed online, as dynamic LPV controllers that converge to the pointwise
optimal gains for constant parameter trajectories~\cite{giessler2025lpv};
the coupled landscape of a shared schedule does not arise in such a setting. For
the same controller class, a receding-horizon reformulation restores
stage-wise convexity~\cite{shakeri2026rhpg}, but certifies a
finite-horizon surrogate rather than the infinite-horizon landscape.
No existing result certifies a convergence rate for a shared schedule
on arbitrary regions of interest, with constants that a user can evaluate.
The letter this document extends provides such a guarantee. The approach is structural: the
benign geometry of the single-plant problem survives weight sharing
in hidden form, and everything proved here follows from identifying that structure.
More explicitly, the contributions are as follows.
\begin{itemize}
\item We prove an exact identity: when the gradient of the cost is
evaluated with the closed-loop covariances at the minimizer, the scheduled cost is
star-convex about that minimizer, on the entire feasible set and for
every choice of vertex basis (\S\ref{sec:identity}).
\item We measure the gap between this substituted gradient and the
true one by a single dimensionless ratio, computable through
Lyapunov equations. Wherever the ratio stays below one on a sublevel
region, gradient descent converges linearly to the optimal schedule
at an explicit rate (\S\ref{sec:rate}).
\item The same ratio can also be used to detect lack of convergence:
at every spurious stationary point this ratio is shown to be larger
than one. A two-plant family with a strictly
suboptimal local minimum shows that such points occur
(\S\ref{sec:identity}, \S\ref{sec:numerics}).
\item Experiments probe the ratio by maximizing it directly: two
benchmarks keep it below one on every region tested, while on an
aeroelastic wing model and on the constructed counterexample the
maximized ratio crosses one at measured cost levels
(\S\ref{sec:numerics}).
\end{itemize}
Supporting proofs and additional experiments appear in the
appendices of this document.
\textit{Notation:} $\lVert\cdot\rVert$ denotes the Euclidean norm for
vectors and the Frobenius norm for matrices and tuples of matrices,
$\lVert\cdot\rVert_2$ the spectral norm, and $\lambda_{\min}(\cdot)$
the smallest eigenvalue of a symmetric matrix; $M \succ 0$
($M \succeq 0$) denotes positive (semi) definiteness, and
$\operatorname{tr}(\cdot)$ the trace;
$\langle X, Y\rangle=\operatorname{tr}(X^\top Y)$ is the trace inner
product, extended to tuples of matrices by summation over the tuple.
% ================= section skeleton (titles tentative) =================
\section{Preliminaries and Problem Formulation}\label{sec:formulation}
We consider a family of frozen-parameter plants on a finite design grid
$\{\rho_1,\dots,\rho_S\}$ with grid weights $\nu_s>0$,
$\sum_{s=1}^{S}\nu_s=1$:
\begin{equation}\label{eq:plant}
x_{t+1} = A(\rho_s)\,x_t + B(\rho_s)\,u_t, \qquad
\mathbb{E}\big[x_0x_0^\top\big]=\Sigma_0 .
\end{equation}
The tensor-product model transformation~\cite{baranyi2004tp} represents
the system matrices in the polytopic form,
\begin{equation}\label{eq:tp}
A(\rho)=\sum_{i=1}^{N} w_i(\rho)\,A_i, \qquad
B(\rho)=\sum_{i=1}^{N} w_i(\rho)\,B_i,
\end{equation}
with vertex systems $(A_i,B_i)$ and continuous weighting functions
satisfying $w_i(\rho)\ge 0$ and $\sum_{i=1}^{N}w_i(\rho)=1$.
The controller interpolates vertex gains $K_i\in\mathbb{R}^{m\times n}$
through the same weighting functions, the
parallel-distributed-compensation structure~\cite{tanaka2001pdc}:
collecting $\mathbf{K}=(K_1,\dots,K_N)$, the deployed gain at the $s$th
grid point is thereby $K_{\mathbf{K}}(\rho_s)=\sum_i w_i(\rho_s)K_i$, applied
as $u_t=-K_{\mathbf{K}}(\rho_s)\,x_t$.
For a gain $K$ that stabilizes plant $s$, i.e.,
$A_{\mathrm{cl}} = A(\rho_s)-B(\rho_s)K$ is Schur stable, the cost for
plant $s$ is,
\begin{equation}\label{eq:Js}
J_s(K)=\mathbb{E}\sum_{t=0}^{\infty}\big(x_t^\top Q x_t
 + u_t^\top R u_t\big)=\operatorname{tr}(P\,\Sigma_0),
\end{equation}
where $P\succ 0$ solves
$P=A_{\mathrm{cl}}^\top P A_{\mathrm{cl}} + Q + K^\top R K$. With the
closed-loop covariance $\Sigma$ solving
$\Sigma = A_{\mathrm{cl}}\Sigma A_{\mathrm{cl}}^\top + \Sigma_0$, the
input weight $H = R + B(\rho_s)^\top P B(\rho_s)$, and the gain residual
$E = HK - B(\rho_s)^\top P A(\rho_s)$, the gradient is
$\nabla_K J_s = 2E\Sigma$~\cite{fazel2018global}; $E=0$ corresponds to
the Riccati optimality condition for plant $s$. We write
$P_s,\Sigma_s,E_s,H_s$ for these matrix constructs at the deployed gain
$K_{\mathbf{K}}(\rho_s)$.
The design problem is to minimize the scheduled cost,
\begin{equation}\label{eq:cost}
\mathcal{J}(\mathbf{K})=\sum_{s=1}^{S}\nu_s\,
J_s\big(K_{\mathbf{K}}(\rho_s)\big)
\end{equation}
over the feasible set
$\mathcal{K}=\{\mathbf{K}: K_{\mathbf{K}}(\rho_s)\ \text{stabilizes
plant } s,\ s=1,\dots,S\}$, an open set. By the chain rule, the gradient
has vertex blocks
\begin{equation}\label{eq:grad}
\nabla_{K_i}\mathcal{J}(\mathbf{K}) =
2\sum_{s=1}^{S}\nu_s\,w_i(\rho_s)\,E_s\Sigma_s ,
\end{equation}
and it is through these shared sums that every vertex gain couples every
plant.
\textit{Assumptions:} $Q\succ0$, $R\succ0$, $\Sigma_0\succ0$;
$\mathcal{K}\neq\emptyset$; and $G_w\succ0$, where
$G_w=\sum_{s=1}^{S}\nu_s\,w(\rho_s)w(\rho_s)^\top$ is the Gram matrix
of the weight vector $w(\rho)=(w_1(\rho),\dots,w_N(\rho))^\top$ on the
grid, singular only when some direction in the vertex space is
invisible on the design grid.
Under these assumptions, $\mathcal{J}$ is coercive on $\mathcal{K}$
and attains its minimum at some
$\mathbf{K}^\star\in\mathcal{K}$ (Appendix~\ref{app:proofs}); we write
$\mathcal{J}^\star=\mathcal{J}(\mathbf{K}^\star)$ and
$\Sigma_s^\star$ for the (state) covariances at the optimal schedule.
\section{The Star Identity}\label{sec:identity}
\begin{definition}[star-convexity
\cite{nesterov2006cubic, necoara2019linear}]\label{def:star}
A differentiable function $f$ is \emph{star-convex} about a minimizer
$x^\star$ if $\langle\nabla f(x),\, x-x^\star\rangle \ge
f(x)-f(x^\star)$ for all $x$ in its domain, and star-convex \emph{with
modulus} $m>0$ if the right-hand side can be strengthened to
$f(x)-f(x^\star)+m\lVert x-x^\star\rVert^2$.
\end{definition}
For $m>0$ the above condition corresponds to the strongly quasar-convex class
of~\cite{hinder2020near}, for which accelerated first-order methods
with matching lower bounds are available.
The gradient \eqref{eq:grad} pairs each residual $E_s$ with the
covariance $\Sigma_s$ of the current schedule. Substituting the
optimizer's covariances yields,
\begin{equation}\label{eq:gstar}
\nabla^{\!\star}_{K_i}\mathcal{J}(\mathbf{K}) =
2\sum_{s=1}^{S}\nu_s\,w_i(\rho_s)\,E_s\Sigma_s^\star ,
\end{equation}
that we refer to as the \emph{starred gradient}. The two fields agree at
$\mathbf{K}^\star$, where $\Sigma_s=\Sigma_s^\star$, and differ
elsewhere only through the covariance mismatch. The starred field
evaluates each residual against the state statistics of the optimal
closed loop rather than those of the current one: the covariances
entering \eqref{eq:grad} describe how the current schedule distributes
the state (not how the optimum does).
Let $\Delta_s = K_{\mathbf{K}}(\rho_s)-K_{\mathbf{K}^\star}(\rho_s)$
denote the increments of the deployed gains, and define,
\begin{equation}\label{eq:QH}
Q_H(\mathbf{K}) = \sum_{s=1}^{S}\nu_s\,
\operatorname{tr}\!\big(\Sigma_s^\star\,\Delta_s^\top H_s\,\Delta_s\big),
\end{equation}
that measures the schedule increment in the geometry of the optimum,
with ``curvature'' $H_s$ and covariance $\Sigma_s^\star$.
Since $\Sigma_s^\star\succeq\Sigma_0$ and $H_s\succeq R$, each summand
of \eqref{eq:QH} satisfies
$\operatorname{tr}(\Sigma_s^\star\Delta_s^\top H_s\Delta_s)\ge
\lambda_{\min}(\Sigma_0)\,\lambda_{\min}(R)\,\lVert\Delta_s\rVert^2$;
and with $D_i=K_i-K_i^\star$ and the positive semidefinite Gram matrix
$\Gamma=[\operatorname{tr}(D_i^\top D_j)]_{i,j}$, the weighted sum
obeys $\sum_{s}\nu_s\lVert\Delta_s\rVert^2=\operatorname{tr}(G_w\Gamma)
\ge\lambda_{\min}(G_w)\,\lVert\mathbf{K}-\mathbf{K}^\star\rVert^2$,
where $G_w$ is the Gram matrix of \S\ref{sec:formulation}. Together
these expressions lead to,
\begin{equation}\label{eq:floor}
Q_H(\mathbf{K}) \ge \underline{\lambda}\,
\lVert\mathbf{K}-\mathbf{K}^\star\rVert^2,
\end{equation}
where $\underline{\lambda}=\lambda_{\min}(G_w)\,\lambda_{\min}(\Sigma_0)\,
\lambda_{\min}(R)$; the constant is positive by the standing
assumption $G_w\succ0$.
\begin{lemma}[star identity]\label{lem:star}
For every $\mathbf{K}\in\mathcal{K}$,
\begin{equation}\label{eq:star}
\big\langle \nabla^{\!\star}\mathcal{J}(\mathbf{K}),\,
\mathbf{K}-\mathbf{K}^\star \big\rangle
= \big(\mathcal{J}(\mathbf{K})-\mathcal{J}^\star\big)
+ Q_H(\mathbf{K}).
\end{equation}
\end{lemma}
\begin{IEEEproof}
For each $s$, both deployed gains stabilize plant $s$, so the
almost-smoothness identity of \cite{fazel2018global} applies:
\[
J_s\big(K_{\mathbf{K}^\star}(\rho_s)\big)
- J_s\big(K_{\mathbf{K}}(\rho_s)\big)
= -2\operatorname{tr}\!\big(\Sigma_s^\star\Delta_s^\top E_s\big)
+ \operatorname{tr}\!\big(\Sigma_s^\star\Delta_s^\top H_s\Delta_s\big).
\]
Multiply the above expression by $\nu_s$ and sum over $s$. Since
$\Delta_s=\sum_i w_i(\rho_s)(K_i-K_i^\star)$, the first terms aggregate
to $-\langle\nabla^{\!\star}\mathcal{J}(\mathbf{K}),
\mathbf{K}-\mathbf{K}^\star\rangle$ by \eqref{eq:gstar}, and the second
terms to $Q_H(\mathbf{K})$. Rearranging leads to \eqref{eq:star}.
\end{IEEEproof}
The identity does not depend on optimality of $\mathbf{K}^\star$:
replacing every occurrence of $\mathbf{K}^\star$ (in $\Delta_s$, in
$\Sigma_s^\star$, and in \eqref{eq:star} itself) by any other
feasible schedule leaves the proof unchanged.
Combining the identity \eqref{eq:star} with the floor \eqref{eq:floor}
yields, for every $\mathbf{K}\in\mathcal{K}$,
\begin{equation}\label{eq:hidden}
\big\langle \nabla^{\!\star}\mathcal{J}(\mathbf{K}),\,
\mathbf{K}-\mathbf{K}^\star \big\rangle
\ge \big(\mathcal{J}(\mathbf{K})-\mathcal{J}^\star\big)
+ \underline{\lambda}\,\lVert\mathbf{K}-\mathbf{K}^\star\rVert^2 ,
\end{equation}
that is Definition~\ref{def:star} for $\mathcal{J}$ with the gradient
replaced by the starred field: star-convexity about $\mathbf{K}^\star$
with modulus $\underline{\lambda}$, valid on all of $\mathcal{K}$, with
no hypothesis beyond the standing assumptions. The corresponding inequality for
$\nabla\mathcal{J}$ may fail, and spurious local minima can exist; in
this sense, the star-convexity is hidden. The extent by which the
inequality transfers from $\nabla^{\!\star}\mathcal{J}$ to
$\nabla\mathcal{J}$ is quantified by the following ratio.
Define the remainder field $\boldsymbol{\mathcal{R}}(\mathbf{K}) =
\tfrac{1}{2}\big(\nabla^{\!\star}\mathcal{J}(\mathbf{K})
-\nabla\mathcal{J}(\mathbf{K})\big)$ and, for
$\mathbf{K}\neq\mathbf{K}^\star$, the \emph{mismatch ratio}
\begin{equation}\label{eq:tau}
\tau(\mathbf{K}) =
\frac{2\,\langle\boldsymbol{\mathcal{R}}(\mathbf{K}),\,
\mathbf{K}-\mathbf{K}^\star\rangle}{Q_H(\mathbf{K})},
\end{equation}
well defined since $Q_H>0$ on
$\mathcal{K}\setminus\{\mathbf{K}^\star\}$ by \eqref{eq:floor}; the
ratio compares the contribution of the covariance mismatch along the
ray to $\mathbf{K}^\star$ against the quadratic term of the identity.
\begin{corollary}\label{cor:transfer}
For every $\mathbf{K}\in\mathcal{K}\setminus\{\mathbf{K}^\star\}$,
\begin{equation}\label{eq:transfer}
\langle \nabla\mathcal{J}(\mathbf{K}),\,
\mathbf{K}-\mathbf{K}^\star \rangle
= \mathcal{J}(\mathbf{K})-\mathcal{J}^\star
+ \big(1-\tau(\mathbf{K})\big)\,Q_H(\mathbf{K}).
\end{equation}
In particular, $\mathcal{J}$ satisfies the star-convexity inequality
at $\mathbf{K}$ if and only if $\tau(\mathbf{K})\le 1$.
\end{corollary}
\begin{IEEEproof}
Substitute $\nabla\mathcal{J}=\nabla^{\!\star}\mathcal{J}
-2\boldsymbol{\mathcal{R}}$ in \eqref{eq:star} and
use \eqref{eq:tau}.
\end{IEEEproof}
\begin{figure}[!ht]
\centering
\includegraphics[width=3.45in]{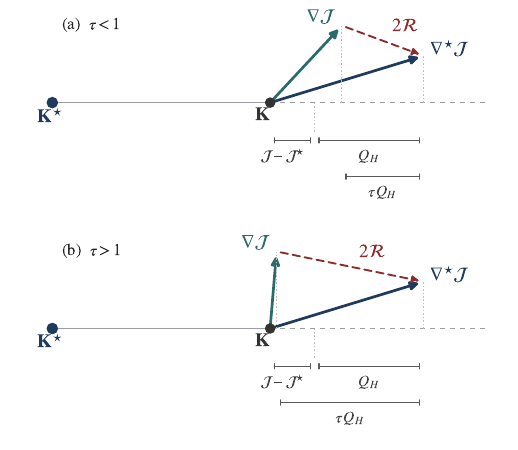}
\caption{The geometry of the identity \eqref{eq:star} and
decomposition \eqref{eq:transfer}. The starred field's
projection onto the ray from $\mathbf{K}^\star$ always equals
$\mathcal{J}-\mathcal{J}^\star+Q_H$; the mismatch shortens the
gradient's projection by $\tau Q_H$. (a) For $\tau<1$ the gradient
retains a projection exceeding $\mathcal{J}-\mathcal{J}^\star$.
(b) For $\tau>1$ the loss exceeds $Q_H$ and the star inequality of
Definition~\ref{def:star} becomes invalid.}
\label{fig:fields}
\end{figure}
Fig.~\ref{fig:fields} shows the two regimes of the decomposition. At
a spurious stationary point the left-hand side
of \eqref{eq:transfer} vanishes while
$\mathcal{J}>\mathcal{J}^\star$, forcing $\tau>1$. Under an
invertible linear reparametrization $\mathbf{K}=L\mathbf{K}'$ of the
vertex gains, with weights $w'=L^\top w$ so that the deployed gains
are unchanged, the quantities $\Delta_s$, $E_s$, $H_s$, $\Sigma_s$,
$\mathcal{J}$, and $Q_H$ are all unchanged, while both gradient
fields transform by $L^\top$ acting blockwise and the increment
$\mathbf{K}-\mathbf{K}^\star$ by $L^{-1}$; every pairing
in \eqref{eq:transfer} is therefore invariant, and so is $\tau$. The
floor constant is not: $G_w$ becomes $L^\top G_w L$, so
$\underline{\lambda}$ depends on the chosen representation.
The identity and the transfer formula are both relative to the
reference schedule: at a stationary point, \eqref{eq:transfer}
forces $\tau>1$ when the cost exceeds the reference value and
$\tau<1$ when it lies below. The guarantees of \S\ref{sec:rate}
accordingly certify convergence to the reference on its component;
that the reference is the global minimizer is a hypothesis of those
results, not a conclusion that the ratio implies.
\section{Convergence Guarantees}\label{sec:rate}
The guarantees of this section concern sublevel components of the
cost. For $c>\mathcal{J}^\star$, let $\mathcal{G}_c$ denote the
connected component of
$\{\mathbf{K}\in\mathcal{K}:\mathcal{J}(\mathbf{K})\le c\}$ containing
$\mathbf{K}^\star$. Each such component is compact: coercivity keeps
the sublevel set bounded and separated from the boundary of
$\mathcal{K}$ (Appendix~\ref{app:proofs}). Let the \emph{critical level}
$c^\dagger$ be the supremum of the levels $c$ for which $\mathcal{G}_c$
contains no stationary point of $\mathcal{J}$ other than
$\mathbf{K}^\star$. For $c<c^\dagger$ the optimum is the only
stationary point in $\mathcal{G}_c$; for $c>c^\dagger$ the component
may contain spurious stationary points, and convergence of descent
methods to $\mathbf{K}^\star$ from all of $\mathcal{G}_c$ is no longer
guaranteed.
\begin{assumption}\label{ass:tau}
There exist $c>\mathcal{J}^\star$ and $\tau_0<1$ such
that $\tau(\mathbf{K})\le\tau_0$ for every
$\mathbf{K}\in\mathcal{G}_c\setminus\{\mathbf{K}^\star\}$.
\end{assumption}
Assumption~\ref{ass:tau} bounds the contribution of the covariance
mismatch by a fixed fraction $\tau_0$ of the quadratic term
in \eqref{eq:transfer}, uniformly over $\mathcal{G}_c$. The hypothesis
can be checked: once the minimizer is available, $\tau$ is computable
at any feasible schedule by Lyapunov solves, and since $\tau>1$ at
every spurious stationary point, the assumption fails on any component
that contains one; in particular a $c$ satisfying the assumption lies
below the critical level. The reach of the hypothesis is captured by
one number: let
$c_\tau=\inf\{c\ge\mathcal{J}^\star:\sup_{\mathcal{G}_c}\tau\ge1\}$,
so that Assumption~\ref{ass:tau} is satisfiable exactly for
$c<c_\tau$, and $c_\tau\le c^\dagger$ in general. \S\ref{sec:numerics} probes the bound by direct maximization of
$\tau$ over sublevel regions, with the protocol in
Appendix~\ref{app:experiments}.
\begin{theorem}\label{thm:pl}
Let Assumption~\ref{ass:tau} hold. Then for every
$\mathbf{K}\in\mathcal{G}_c$,
\begin{equation}\label{eq:pl}
\lVert\nabla\mathcal{J}(\mathbf{K})\rVert^2 \ge
\mu\,\big(\mathcal{J}(\mathbf{K})-\mathcal{J}^\star\big),
\qquad
\mu = 4\,(1-\tau_0)\,\underline{\lambda}.
\end{equation}
\end{theorem}
\begin{IEEEproof}
Fix $\mathbf{K}\neq\mathbf{K}^\star$ and abbreviate
$g=\mathcal{J}(\mathbf{K})-\mathcal{J}^\star$ and
$\mathbf{\Delta}=\mathbf{K}-\mathbf{K}^\star$.
By \eqref{eq:transfer}, Assumption~\ref{ass:tau}, and the
floor \eqref{eq:floor},
\[
\langle \nabla\mathcal{J}(\mathbf{K}),\,\mathbf{\Delta}\rangle
\ge g + (1-\tau_0)\,\underline{\lambda}\,\lVert\mathbf{\Delta}\rVert^2 .
\]
The Cauchy--Schwarz inequality bounds the left side by
$\lVert\nabla\mathcal{J}(\mathbf{K})\rVert\,\lVert\mathbf{\Delta}\rVert$,
and the arithmetic--geometric mean inequality bounds the right side
below by
$2\sqrt{(1-\tau_0)\,\underline{\lambda}\,g}\;\lVert\mathbf{\Delta}\rVert$.
Dividing by $\lVert\mathbf{\Delta}\rVert$ and squaring implies the claim;
at $\mathbf{K}=\mathbf{K}^\star$ both sides vanish.
\end{IEEEproof}
Inequality \eqref{eq:pl} is a gradient-dominance condition of
Polyak--{\L}ojasiewicz type \cite{karimi2016linear} on $\mathcal{G}_c$,
with a constant assembled from the certified fraction $\tau_0$ and the
spectral floor $\underline{\lambda}$. Such a condition yields linear
convergence of gradient descent provided the iterates remain in the
region where it holds; the next result verifies this and fixes an
admissible step size.
\begin{theorem}\label{thm:rate}
Let Assumption~\ref{ass:tau} hold and let
$\mathbf{K}^{(0)}\in\mathcal{G}_c$. There exists $\bar\eta>0$,
depending only on $\mathcal{G}_c$ and a compact neighborhood of it,
such that for every
$\eta\in(0,\bar\eta]$ the iteration
$\mathbf{K}^{(t+1)}=\mathbf{K}^{(t)}
-\eta\,\nabla\mathcal{J}(\mathbf{K}^{(t)})$
remains in $\mathcal{G}_c$ and satisfies,
\begin{equation}\label{eq:rate}
\mathcal{J}(\mathbf{K}^{(t)})-\mathcal{J}^\star \le
\Big(1-\tfrac{\eta\mu}{2}\Big)^{\!t}
\big(\mathcal{J}(\mathbf{K}^{(0)})-\mathcal{J}^\star\big),
\end{equation}
with $\mathbf{K}^{(t)}\to\mathbf{K}^\star$.
\end{theorem}
\begin{IEEEproof}
The cost is analytic on $\mathcal{K}$, since the Lyapunov solutions
in \eqref{eq:Js} depend analytically on the gain (Appendix~\ref{app:proofs}); on a compact neighborhood of $\mathcal{G}_c$
contained in $\mathcal{K}$ its gradient is therefore Lipschitz with
some constant $L$. Now choose $\varepsilon>0$ such that every point within
distance $\varepsilon$ of $\mathcal{G}_c$ lies in this neighborhood.
Two other constants of the component control the iteration: $G$, the
largest gradient norm on $\mathcal{G}_c$, so that no step moves the
iterate farther than $\eta G$; and $d$, the distance from
$\mathcal{G}_c$ to the other components of the sublevel set, positive
as $\mathcal{G}_c$ is compact. Take
$\bar\eta=\min\{1/L,\,\varepsilon/G,\,d/(2G)\}$. Each step then moves
at most $\eta G\le\varepsilon$, so the segment between consecutive
iterates stays in the neighborhood where the descent lemma applies:
for $\eta\le 1/L$,
$\mathcal{J}(\mathbf{K}^{(t+1)})\le\mathcal{J}(\mathbf{K}^{(t)})
-\tfrac{\eta}{2}\lVert\nabla\mathcal{J}(\mathbf{K}^{(t)})\rVert^2$, so
the cost never increases and the iterates remain in the sublevel set;
since each step also moves at most $d/2$, the iterates cannot cross
into another component and remain in $\mathcal{G}_c$.
Inserting \eqref{eq:pl} into the descent inequality yields the
contraction \eqref{eq:rate}. The cost gap therefore converges to zero;
every accumulation point of the sequence, bounded because
$\mathcal{G}_c$ is compact, attains $\mathcal{J}^\star$; and
$\mathbf{K}^\star$ is the unique such point in $\mathcal{G}_c$, so the
sequence converges to it. Explicit expressions for $L$, $G$, and $d$ are derived in
Appendix~\ref{app:proofs}.
\end{IEEEproof}
\section{Numerical Studies}\label{sec:numerics}
Three plant families and one constructed example are examined, all
with uniform grid weights $\nu_s=1/S$. The first is a second-order
LPV family with affine
dependence on a single parameter $\rho\in[0,1]$, as in \eqref{eq:tp},
with vertex data
\begin{equation}\label{eq:sysA}
A_1=\begin{bmatrix}0.90 & 0.15\\ -0.10 & 0.85\end{bmatrix},\quad
A_2=\begin{bmatrix}1.05 & 0.35\\ -0.10 & 0.97\end{bmatrix},\quad
B_1=\begin{bmatrix}1.0\\ 0.2\end{bmatrix},\quad
B_2=\begin{bmatrix}0.7\\ 0.6\end{bmatrix},
\end{equation}
weighting functions $w(\rho)=(1-\rho,\;\rho)^\top$, $S=21$, and
$Q=\Sigma_0=I_2$, $R=1$; the vertex gains
$K_1,K_2\in\mathbb{R}^{1\times2}$ give four decision variables. The
family is open-loop unstable near $\rho=1$, with spectral radius
$1.03$ at that vertex.
The second is a Duffing oscillator with softening cubic stiffness
driven through a first-order actuator,
\begin{equation}\label{eq:duffing}
m\ddot{x} + c\dot{x} + k_0 x - k_1 x^3 = u_a, \qquad
\tau\dot{u}_a = u - u_a,
\end{equation}
with $m=k_0=k_1=1$, $c=0.4$, $\tau=0.2$, and states
$(x,\dot{x},u_a)$: position, velocity, and actuator state. On
$|x|\le1.5$ the substitution $\rho=x^2$ yields an exact quasi-LPV
representation with stiffness
$k_0-k_1\rho$ \cite{tanaka1998chaos}, two vertices with
$w(\rho)=(1-\rho/2.25,\;\rho/2.25)^\top$. Euler discretization with
step $0.05$, $S=21$, $Q=\Sigma_0=I_3$, $R=1$; six decision variables.
The frozen plants are open-loop unstable for $\rho>1$.
The third is the tensor-product model of an aeroelastic wing
section \cite{takarics2014wing}: fourteen states, one input, and four
vertex gains produced by the TP model transformation, on a grid of
$S=51$ points, with diagonal $Q$ weighting two states by $10^4$ and
the remainder by one, $R=1$, and $\Sigma_0=I_{14}$, for $56$ decision
variables. The floor constants are $\lambda_{\min}(G_w)=0.183$ for
the first two systems and $0.035$ for the wing. A fourth example,
constructed to contain a spurious local minimum, is introduced with
the experiment that uses it.
We first verify the identity \eqref{eq:star} numerically. For each
system the minimizer is located by gradient descent with backtracking
line search from a pointwise Riccati warm start; sample points are
then generated by drawing random directions from $\mathbf{K}^\star$
under fixed seeds and bisecting along each ray onto the level sets
between $c=1.1\,\mathcal{J}^\star$ and $8\,\mathcal{J}^\star$. At
every sample the two sides of \eqref{eq:star} are evaluated
independently. The maximum relative discrepancy is
$1.3\times10^{-14}$ over 240 samples for the second-order system,
$2.7\times10^{-14}$ over 240 samples for the Duffing system, and
$5.1\times10^{-13}$ over 144 samples for the wing, with median
discrepancies near $10^{-15}$: the identity holds at roundoff level.
Since \eqref{eq:star} is valid for any feasible reference schedule,
the check is insensitive to the accuracy with which the minimizer
itself is located.
\begin{table}[!ht]
\caption{Largest sampled mismatch ratio $\tau$ per level.}
\label{tab:tau}
\centering
\begin{tabular}{cccc}
\hline
$c/\mathcal{J}^\star$ & $\max\tau$, second-order & $\max\tau$, Duffing
& $\max\tau$, wing\\
\hline
$1.1$ & $0.520$ & $0.774$ & $0.330$\\
$1.5$ & $0.550$ & $0.858$ & $0.413$\\
$2.0$ & $0.487$ & $0.847$ & $-0.042$\\
$3.0$ & $0.220$ & $0.573$ & $-0.998$\\
$5.0$ & $-1.50$ & $0.182$ & $-7.71$\\
$8.0$ & $-4.95$ & $-1.26$ & $-15.4$\\
\hline
\end{tabular}
\end{table}
The mismatch ratio \eqref{eq:tau} is measured on the same sampled
components, under the same protocol.
Table~\ref{tab:tau} reports the largest sampled value of $\tau$ at
each level, over 240 level-set samples for each
of the first two systems and 144 for the wing (40 and 24 rays per
level, fixed seeds; radial shells probe the interior separately); the wing minimizer is located to gradient norm
$8\times10^{-3}$. On every sampled component the maximum lies below
one: $0.55$ for the second-order system, $0.86$ for the Duffing
system, and $0.41$ for the wing, attained near
$c=1.5\,\mathcal{J}^\star$ in all three cases. At the highest
levels the sampled maxima turn negative: far
from the minimizer the covariance mismatch enters \eqref{eq:transfer}
with favorable sign. A sampled maximum estimates the supremum of
\eqref{eq:tau} from below, so the ratio is also maximized directly:
projected gradient ascent on $\tau$ within $\{\mathcal{J}\le c\}$,
from multiple starts per level under fixed seeds
(protocol in Appendix~\ref{app:experiments}), with iterates and
radial shells covering the interior of the component rather than its
level sets alone. The attack raises the largest ratio to $0.79$ for
the second-order system and $0.995$ for the Duffing system, both at
interior points near $3\,\mathcal{J}^\star$, and crosses the
threshold on the wing: the largest ratio found grows from $0.98$ at
$c=1.1\,\mathcal{J}^\star$ to $1.02$ at $1.5\,\mathcal{J}^\star$
and $1.52$ at $5\,\mathcal{J}^\star$,
with a violating gain of cost $1.37\,\mathcal{J}^\star$ connected
to $\mathbf{K}^\star$ within its own sublevel set.
Assumption~\ref{ass:tau} therefore holds with a wide margin on the
first system, with almost no margin on the second, and fails beyond
$1.37\,\mathcal{J}^\star$ on the wing example.
\begin{figure}[!ht]
\centering
\includegraphics[width=3.45in]{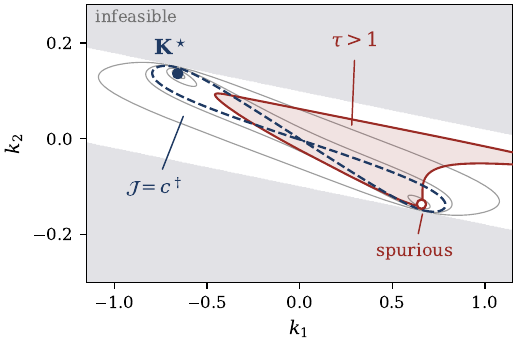}
\caption{The mirror example \eqref{eq:mirror} in the gain plane: cost
contours (gray), the critical-level contour $\mathcal{J}=c^\dagger$
(dashed), and the region $\tau>1$ (shaded), which contains the
spurious minimizer and its basin, reaches into the component of
$\mathbf{K}^\star$ below the critical level, and leaves a
neighborhood of $\mathbf{K}^\star$ clear.}
\label{fig:mirror}
\end{figure}
The fourth example reproduces the crossing in the plane, where every
quantity can be computed exactly. Two frozen plants share the state
matrix and differ only in the sign of the input matrix, namely,
\begin{equation}\label{eq:mirror}
A = \begin{bmatrix}0.9467 & -0.0018\\ 1.1218 & 0.7492\end{bmatrix},
\qquad
B_1 = -B_2 = \begin{bmatrix}-0.0610\\ -1.6287\end{bmatrix},
\end{equation}
with $\nu=(0.51,\,0.49)$, a single gain $K\in\mathbb{R}^{1\times2}$
deployed on both plants, and $Q=\Sigma_0=I_2$, $R=1$. For $S=N$ the
deployed gains decouple and no spurious minimum can arise, so a trap
needs more plants than vertices; whether one exists for a schedule
with $N\ge2$ is open (note that the interpolated family loses
controllability at $\rho=\tfrac12$, between the two grid points). The
gain
$K^0=[-0.6577,\ 0.1363]$ stabilizes both plants, as does its mirror
image $-K^0$. Gradient descent from $K^0$ converges to the global
minimum at $\mathcal{J}^\star=113.51$; descent from $-K^0$ converges
to a spurious local minimum at $\mathcal{J}=116.83$, located to
gradient norm $2\times10^{-5}$. The ratio measured at the spurious
point is $\tau=1.03451$, agreeing to six significant digits with the
value $1+(\mathcal{J}-\mathcal{J}^\star)/Q_H$
that \eqref{eq:transfer} forces at any stationary point other than
the minimizer; over 228 samples in a neighborhood of the spurious
minimizer, $\tau$ reaches $3.85$. Fig.~\ref{fig:mirror} shows the
region $\tau>1$ in the gain plane: it contains the spurious basin and
the saddle corridor, and it does not stop at the boundary of the
component of $\mathbf{K}^\star$. The level $c_\tau$ of
\S\ref{sec:rate} is computed exactly here. Newton's
method locates the ridge saddle at cost $c^\dagger=153.16$, and
minimizing $\mathcal{J}$ along the curve $\tau=1$ inside the
component gives $c_\tau=132.05$: Assumption~\ref{ass:tau} fails at
$47\%$ of the level range between $\mathcal{J}^\star$ and
$c^\dagger$, while spurious stationarity enters only at
$c^\dagger$. The hypothesis is
thus strictly stronger than the absence of spurious stationary
points, and it must be verified rather than inferred from the
topology. We leave open whether the fraction
$(c_\tau-\mathcal{J}^\star)/(c^\dagger-\mathcal{J}^\star)$, confined
to $[0,1]$ by the preceding inequality, admits a universal positive
lower bound; this example supplies the single data point $0.47$.
\begin{figure}[!ht]
\centering
\includegraphics[width=3.45in]{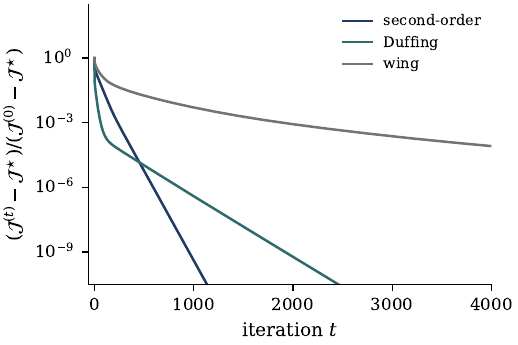}
\caption{Constant-step gradient descent on the three benchmarks.}
\label{fig:convergence}
\end{figure}
Theorem~\ref{thm:rate} is tested by running the
constant-step iteration on all three systems, initialized on the
$2\,\mathcal{J}^\star$ level set under fixed-seed rays, with the step
chosen as the largest power of two that never increases the cost:
$\eta=2^{-7}$, $2^{-9}$, and $1.5\times10^{-8}$, respectively.
Fig.~\ref{fig:convergence} shows the normalized gap: all three
trajectories become linear on semilogarithmic axes, with the wing's
shallow slope set by its Hessian spread of roughly 360 to one, which
enters the guarantee through the product $\eta\mu$; the wing's gap
has fallen to $8.2\times10^{-5}$ after $4000$ iterations. The wing
example converges although Assumption~\ref{ass:tau} fails on its
component: the hypothesis is sufficient for the guaranteed rate, not
necessary for convergence.
\section{Conclusion}\label{sec:conclusion}
This paper has identified an exact mechanism behind the reliability of
policy gradients for gain-scheduled LQR: when the gradient of the cost
is evaluated with the closed-loop covariances of the minimizer, the
scheduled cost satisfies a star-convexity inequality about the
minimizer on the entire feasible set. The ``alignment'' between this
hidden feature and a convergence guarantee of descent-type algorithms
is captured by a single dimensionless quantity, the mismatch ratio,
invariant to the choice of vertex basis, computable once the minimizer
is located, and greater than one at every spurious stationary point.
Bounding the ratio on a sublevel component yields gradient dominance
and a linear rate of convergence for first-order methods.
Two technical issues remain open: how to bound this ratio a priori, in
particular whether the certifiable fraction of the range below the
critical level admits a universal lower bound, and how to certify a
level below the critical one without knowing the minimizer.
\section*{Acknowledgment}
The research of the authors has been supported by the
2024-1.2.3-HU-RIZONT-2024-00030 project.
% =====================================================
\appendix
\section{Technical Proofs}\label{app:proofs}
\subsection{Coercivity and Existence of the Minimizer}\label{app:coercive}
\begin{proposition}\label{prop:coercive}
Under the standing assumptions of \S\ref{sec:formulation},
$\mathcal{J}$ is coercive on $\mathcal{K}$: along any sequence in
$\mathcal{K}$ that either is unbounded or converges to a point of
$\partial\mathcal{K}$, $\mathcal{J}\to\infty$. Consequently every
sublevel set of $\mathcal{J}$ is compact, and $\mathcal{J}$ attains
its minimum at some $\mathbf{K}^\star\in\mathcal{K}$.
\end{proposition}
\begin{IEEEproof}
First, a pointwise lower bound, as in the single-plant
analysis~\cite{fazel2018global}. For any gain $K$ stabilizing plant
$s$, the Lyapunov solution satisfies $P\succeq Q+K^\top RK$, since
$P=\sum_{t\ge0}(A_{\mathrm{cl}}^t)^\top(Q+K^\top RK)
A_{\mathrm{cl}}^t$ and the $t=0$ term alone gives the bound; hence
\[
J_s(K)\;\ge\;\lambda_{\min}(\Sigma_0)
\big(\operatorname{tr}Q+\lambda_{\min}(R)\,\lVert K\rVert^2\big).
\]
Summing with weights $\nu_s$ and applying the Gram bound of
\S\ref{sec:identity} to
$\sum_s\nu_s\lVert K_{\mathbf{K}}(\rho_s)\rVert^2$ gives
\[
\mathcal{J}(\mathbf{K})\;\ge\;
\lambda_{\min}(\Sigma_0)\,\lambda_{\min}(R)\,
\lambda_{\min}(G_w)\,\lVert\mathbf{K}\rVert^2
+ \lambda_{\min}(\Sigma_0)\operatorname{tr}Q,
\]
so $\mathcal{J}\to\infty$ along any unbounded sequence. This is the
one place where coercivity uses $G_w\succ0$, and the condition is not
removable: $\mathcal{J}$ is constant along every direction in the
kernel of $G_w$.
Second, blow-up at the boundary. Let
$\mathbf{K}_j\to\mathbf{K}_\partial\in\partial\mathcal{K}$, so
that for some $s$ the limiting closed loop $A_\partial$ has spectral
radius at least one. Suppose, for contradiction, that $J_s$ stays
bounded along the sequence. Then $\operatorname{tr}(P_j\Sigma_0)$ is
bounded, so the matrices $P_j\succeq Q\succ0$ are bounded and a
subsequence converges to some $\bar P\succ0$. Passing to the limit
in the Lyapunov equation gives
$\bar P=A_\partial^\top\bar P A_\partial+Q
+K_\partial^\top RK_\partial$ with $\bar P\succ0$ and
$Q+K_\partial^\top RK_\partial\succ0$, which by the Lyapunov
stability theorem forces $A_\partial$ to be Schur, a contradiction.
Hence $J_s\to\infty$, and so does $\mathcal{J}$.
Coercivity bounds every sublevel set and excludes its limit points on
$\partial\mathcal{K}$, so each is closed in the ambient space and
therefore compact; in particular each lies at positive distance from
$\partial\mathcal{K}$, the form used in \S\ref{sec:rate}.
Continuity of $\mathcal{J}$ on the nonempty set $\mathcal{K}$ and
the Weierstrass theorem then give the minimizer.
\end{IEEEproof}
\subsection{Analyticity and the Descent Inequality}\label{app:analytic}
\begin{proposition}\label{prop:analytic}
(i) On $\mathcal{K}$, the maps $\mathbf{K}\mapsto P_s$ and
$\mathbf{K}\mapsto\Sigma_s$ are real-analytic for every $s$, and so
is $\mathcal{J}$. (ii) Let $C\subset\mathcal{K}$ be compact and let
$0<\varepsilon<\operatorname{dist}(C,\partial\mathcal{K})$. Then
the closed $\varepsilon$-inflation $N$ of $C$ is a compact subset of
$\mathcal{K}$, the constant
$L:=\sup_{N}\lVert\nabla^2\mathcal{J}\rVert_2$ is finite, and for
every $\mathbf{K}\in C$ and every $\mathbf{K}'$ with
$\lVert\mathbf{K}'-\mathbf{K}\rVert\le\varepsilon$,
\[
\mathcal{J}(\mathbf{K}')\;\le\;\mathcal{J}(\mathbf{K})
+\big\langle\nabla\mathcal{J}(\mathbf{K}),\,
\mathbf{K}'-\mathbf{K}\big\rangle
+\tfrac{L}{2}\,\lVert\mathbf{K}'-\mathbf{K}\rVert^2 .
\]
\end{proposition}
\begin{IEEEproof}
For (i), vectorizing the Lyapunov equations gives the explicit
formulas
$\operatorname{vec}P=(I-A_{\mathrm{cl}}^\top\!\otimes
A_{\mathrm{cl}}^\top)^{-1}\operatorname{vec}(Q+K^\top RK)$ and
$\operatorname{vec}\Sigma=(I-A_{\mathrm{cl}}\otimes
A_{\mathrm{cl}})^{-1}\operatorname{vec}\Sigma_0$. On the
stabilizing set the determinant
$\det(I-A_{\mathrm{cl}}\otimes A_{\mathrm{cl}})
=\prod_{i,j}(1-\lambda_i\lambda_j)$~\cite{horn1991topics}, where
the $\lambda_i$ are the eigenvalues of $A_{\mathrm{cl}}$, does not
vanish, since every
$|\lambda_i|<1$; each entry of $P$ and $\Sigma$ is therefore a
rational function of the entries of $K$ with nonvanishing
denominator, hence real-analytic. The cost
$J_s=\operatorname{tr}(P\Sigma_0)$ inherits analyticity, and
$\mathcal{J}$ is a finite weighted sum of analytic functions composed
with the linear maps $\mathbf{K}\mapsto K_{\mathbf{K}}(\rho_s)$.
For (ii), every point of $N$ lies within $\varepsilon$ of $C$ and
therefore at positive distance from $\partial\mathcal{K}$, so $N$ is
a compact subset of $\mathcal{K}$ on which the continuous
$\nabla^2\mathcal{J}$ is bounded, making $L$ finite. The segment
from $\mathbf{K}$ to $\mathbf{K}'$ lies in the $\varepsilon$-ball
around $\mathbf{K}\in C$, hence in $N$, and Taylor's theorem with
integral remainder along this segment gives the inequality.
\end{IEEEproof}
\subsection{Sublevel Bounds and the Gradient Constant}\label{app:levels}
\begin{lemma}[level bounds]\label{lem:levels}
Let $\ell>0$ and let $\mathbf{K}$ satisfy
$\mathcal{J}(\mathbf{K})\le\ell$. Then for every $s$, with all
quantities at the deployed gain,
\[
\operatorname{tr}P_s\le\frac{\ell}{\nu_s\,\lambda_{\min}(\Sigma_0)},
\qquad
\operatorname{tr}\Sigma_s\le\frac{\ell}{\nu_s\,\lambda_{\min}(Q)},
\qquad
\lVert K_{\mathbf{K}}(\rho_s)\rVert^2\le
\frac{\ell}{\nu_s\,\lambda_{\min}(\Sigma_0)\,\lambda_{\min}(R)} .
\]
\end{lemma}
\begin{IEEEproof}
Since the pointwise costs are nonnegative,
$\nu_s J_s\le\mathcal{J}\le\ell$. The first bound follows from
$J_s=\operatorname{tr}(P_s\Sigma_0)\ge
\lambda_{\min}(\Sigma_0)\operatorname{tr}P_s$; the second from the
alternative trace form
$J_s=\operatorname{tr}\big((Q+K^\top RK)\Sigma_s\big)\ge
\lambda_{\min}(Q)\operatorname{tr}\Sigma_s$; the third from
$P_s\succeq Q+K^\top RK$ as in
Proposition~\ref{prop:coercive}.
\end{IEEEproof}
\begin{corollary}[gradient bound]\label{cor:gradbound}
Write $p_s(\ell)$ and $\sigma_s(\ell)$ for the first two bounds of
Lemma~\ref{lem:levels}, set
$k_s(\ell)=\big(\ell/(\nu_s\lambda_{\min}(\Sigma_0)
\lambda_{\min}(R))\big)^{1/2}$, and define
$h_s(\ell)=\lVert R\rVert_2
+\lVert B(\rho_s)\rVert_2^2\,p_s(\ell)$ and
$e_s(\ell)=h_s(\ell)\,k_s(\ell)
+\lVert B(\rho_s)\rVert_2\lVert A(\rho_s)\rVert_2\,p_s(\ell)$.
Then on $\{\mathcal{J}\le\ell\}$,
\[
\lVert\nabla\mathcal{J}(\mathbf{K})\rVert\;\le\;
G(\ell):=2\sum_{s=1}^{S}\nu_s\,e_s(\ell)\,\sigma_s(\ell).
\]
\end{corollary}
\begin{IEEEproof}
$\lVert H_s\rVert_2\le h_s$ and $\lVert E_s\rVert\le e_s$ follow
from the definitions of $H_s$ and $E_s$ with
$\lVert P_s\rVert_2\le\operatorname{tr}P_s$. In the gradient
tuple \eqref{eq:grad}, each grid point contributes at most
$2\nu_s\lVert w(\rho_s)\rVert_2\lVert E_s\Sigma_s\rVert
\le2\nu_s e_s\sigma_s$ to the norm, using
$\lVert w(\rho)\rVert_2\le\sum_i w_i(\rho)=1$ and
$\lVert\Sigma_s\rVert_2\le\operatorname{tr}\Sigma_s$; the
triangle inequality over $s$ gives the claim.
\end{IEEEproof}
Every quantity on the right of Lemma~\ref{lem:levels} and
Corollary~\ref{cor:gradbound} is computable from the data
$(A_s,B_s,Q,R,\Sigma_0,\nu,\ell)$. The constant $G$ of
Theorem~\ref{thm:rate} is $G(c)$.
\subsection{Perturbation Lemmas}\label{app:perturb}
Fix a plant; the index $s$ is suppressed. For $\zeta>0$, consider
gains $K$ with $A_{\mathrm{cl}}$ Schur and $J(K)\le\zeta$, and set
\[
p(\zeta)=\tfrac{\zeta}{\lambda_{\min}(\Sigma_0)},\quad
\sigma(\zeta)=\tfrac{\zeta}{\lambda_{\min}(Q)},\quad
k(\zeta)=\Big(\tfrac{\zeta}{\lambda_{\min}(\Sigma_0)
\lambda_{\min}(R)}\Big)^{1/2},
\]
\[
a(\zeta)=\lVert A\rVert_2+\lVert B\rVert_2\,k(\zeta),\quad
h(\zeta)=\lVert R\rVert_2+\lVert B\rVert_2^2\,p(\zeta),\quad
g(\zeta)=\tfrac{\zeta}{\lambda_{\min}(\Sigma_0)
\lambda_{\min}(Q)}.
\]
\begin{lemma}[perturbation of the Lyapunov solutions]\label{lem:perturb}
Let $K$ and $K'$ both satisfy the conditions above. Then
\[
\lVert\Sigma(K')-\Sigma(K)\rVert_2\le
c_\Sigma(\zeta)\,\lVert K'-K\rVert,
\qquad
c_\Sigma(\zeta)=2\,g(\zeta)\,\lVert B\rVert_2\,
a(\zeta)\,\sigma(\zeta),
\]
\[
\lVert P(K')-P(K)\rVert_2\le c_P(\zeta)\,\lVert K'-K\rVert,
\qquad
c_P(\zeta)=2\,g(\zeta)\big(a(\zeta)\lVert B\rVert_2\,p(\zeta)
+\lVert R\rVert_2\,k(\zeta)\big).
\]
\end{lemma}
\begin{IEEEproof}
For a Schur matrix $F$, let
$\mathcal{T}_F(X)=\sum_{t\ge0}F^tX(F^\top)^t$, the solution of
$Y=FYF^\top+X$. Two facts are used throughout. First, the Gramian
bound: $\mathcal{T}_{A'_{\mathrm{cl}}}(I)\preceq
\Sigma(K')/\lambda_{\min}(\Sigma_0)$ and
$\mathcal{T}_{A_{\mathrm{cl}}'^{\top}}(I)\preceq
P(K')/\lambda_{\min}(Q)$, and since
$\operatorname{tr}\Sigma(K')\le\sigma(\zeta)$ and
$\operatorname{tr}P(K')\le p(\zeta)$ by the pointwise form of
Lemma~\ref{lem:levels}, both Gramians have spectral norm at most
$g(\zeta)$. Second, for symmetric $X$ with positive part $X_+$ and
negative part $X_-$,
\[
\lVert\mathcal{T}_F(X)\rVert_2\le
\lVert\mathcal{T}_F(X_+)\rVert_2+\lVert\mathcal{T}_F(X_-)\rVert_2
\le\operatorname{tr}\mathcal{T}_F(X_+)
+\operatorname{tr}\mathcal{T}_F(X_-)
\le\lVert\mathcal{T}_{F^\top}(I)\rVert_2\,\lVert X\rVert_*,
\]
by cyclicity of the trace. Now subtract the two covariance equations:
$\Sigma'-\Sigma=\mathcal{T}_{A'_{\mathrm{cl}}}(M)$ with
$M=\Delta A\,\Sigma A_{\mathrm{cl}}'^\top
+A_{\mathrm{cl}}\Sigma\,\Delta A^\top$ and
$\Delta A=-B(K'-K)$. Since
$\lVert M\rVert_*\le2\lVert B\rVert_2\,a(\zeta)\,
\sigma(\zeta)\lVert K'-K\rVert$, the two facts give the first
claim. The value matrices subtract the same way, with
$M=\Delta A^\top PA'_{\mathrm{cl}}
+A_{\mathrm{cl}}^\top P\Delta A
+\Delta K^\top RK'+K^\top R\,\Delta K$, whose nuclear norm is at
most $2\big(a(\zeta)\lVert B\rVert_2\,p(\zeta)
+\lVert R\rVert_2\,k(\zeta)\big)\lVert K'-K\rVert$; the same
two facts give the second claim.
\end{IEEEproof}
\begin{corollary}[residual perturbation]\label{cor:eperturb}
On the same set,
$\lVert E(K')-E(K)\rVert\le c_E(\zeta)\lVert K'-K\rVert$ with
\[
c_E(\zeta)=h(\zeta)
+\big(\lVert B\rVert_2^2\,k(\zeta)
+\lVert B\rVert_2\lVert A\rVert_2\big)\,c_P(\zeta).
\]
\end{corollary}
\begin{IEEEproof}
$E'-E=B^\top\!\Delta P\,B\,K'+H\,\Delta K
-B^\top\!\Delta P\,A$; bound the three terms by
$\lVert B\rVert_2^2c_Pk$, $h\,\lVert\Delta K\rVert$, and
$\lVert B\rVert_2\lVert A\rVert_2c_P$ respectively.
\end{IEEEproof}
Every constant above is an explicit function of
$(\lVert A\rVert_2,\lVert B\rVert_2,\lambda_{\min}(Q),
\lambda_{\min}(R),\lambda_{\min}(\Sigma_0),\zeta)$, evaluated per
plant.
\subsection{The Lipschitz and Separation Constants}\label{app:LGd}
\begin{lemma}[margin inflation]\label{lem:margin}
Fix a plant, let $\zeta>0$, and let
$c_J(\eta):=c_P(\eta)\operatorname{tr}\Sigma_0$ denote the
pointwise cost-Lipschitz constant obtained from
Lemma~\ref{lem:perturb}. If $J(K)\le\zeta$ and
$\lVert K'-K\rVert<\varepsilon_0(\zeta):=\zeta/c_J(2\zeta)$,
then $K'$ is stabilizing and $J(K')\le2\zeta$.
\end{lemma}
\begin{IEEEproof}
Along the segment from $K$ to $K'$, the set where $J<2\zeta$ is open
and contains the start. If the cost first reached $2\zeta$ at a point
at distance $\delta^*$, both endpoints of the traversed sub-segment
would have cost at most $2\zeta$, so Lemma~\ref{lem:perturb} would
give $2\zeta-\zeta\le c_J(2\zeta)\,\delta^*$, that is,
$\delta^*\ge\varepsilon_0(\zeta)$; hence no crossing occurs before
$\varepsilon_0(\zeta)$, and stability persists because the cost
blows up at the boundary of the stabilizing set
(Proposition~\ref{prop:coercive}).
\end{IEEEproof}
\begin{proposition}[the constant $L$]\label{prop:L}
Let $c\in(\mathcal{J}^\star,c^\dagger)$, set $\zeta_s=c/\nu_s$
and $\bar\zeta_s=2\zeta_s$, and let
$\varepsilon:=\min_s\varepsilon_0(\zeta_s)/2$. On the
$\varepsilon$-inflation of $\mathcal{G}_c$, every pointwise cost
satisfies $J_s\le\bar\zeta_s$, and $\nabla\mathcal{J}$ is
Lipschitz with the explicit constant
\[
L=2\sum_{s=1}^{S}\nu_s\Big(c_{E,s}(\bar\zeta_s)\,
\sigma_s(\bar\zeta_s)+e_s(\bar\zeta_s)\,
c_{\Sigma,s}(\bar\zeta_s)\Big),
\qquad
e_s(\zeta):=h_s(\zeta)k_s(\zeta)
+\lVert B(\rho_s)\rVert_2\lVert A(\rho_s)\rVert_2\,p_s(\zeta).
\]
\end{proposition}
\begin{IEEEproof}
The cost bound on the inflation is Lemma~\ref{lem:margin} applied per
plant; in particular the inflation is a compact subset of
$\mathcal{K}$, so Proposition~\ref{prop:analytic}(ii) applies with
this $\varepsilon$. For the Lipschitz constant, write the per-plant
field $M_s=E_s\Sigma_s$ and subtract at two schedules in the
inflation:
$M_s'-M_s=(E_s'-E_s)\Sigma_s'+E_s(\Sigma_s'-\Sigma_s)$, so
$\lVert M_s'-M_s\rVert\le
\big(c_{E,s}\sigma_s+e_s\,c_{\Sigma,s}\big)
\lVert\mathbf{K}'-\mathbf{K}\rVert$, using
$\lVert K_{\mathbf{K}'}(\rho_s)-K_{\mathbf{K}}(\rho_s)\rVert
\le\lVert w(\rho_s)\rVert_2
\lVert\mathbf{K}'-\mathbf{K}\rVert
\le\lVert\mathbf{K}'-\mathbf{K}\rVert$. Summing the gradient
tuple as in Corollary~\ref{cor:gradbound} gives the claim.
\end{IEEEproof}
\begin{proposition}[the distance $d$]\label{prop:d}
For $c\in(\mathcal{J}^\star,c^\dagger)$, the distance from
$\mathcal{G}_c$ to the remainder of the sublevel set satisfies
\[
d\;\ge\;\frac{2\,(c^\dagger-c)}{G(c^\dagger)},
\]
with $G$ the gradient bound of Corollary~\ref{cor:gradbound}.
\end{proposition}
\begin{IEEEproof}
Let $x\in\mathcal{G}_c$ and $y$ in another component, and consider
the straight segment between them. The pass value between the two
components, the infimum over connecting paths of the maximum of
$\mathcal{J}$, is attained at a stationary point by the mountain-pass
theorem~\cite{ambrosetti1973dual}, whose Palais--Smale condition
holds here because $\mathcal{J}$ is coercive with boundary blow-up;
that stationary point is distinct from $\mathbf{K}^\star$ because
its value is at least $c>\mathcal{J}^\star$, and it lies in the
merged component, so the pass value is at least $c^\dagger$ by the
definition of the critical level. The cost along the segment
therefore rises from at most $c$ to at least $c^\dagger$ and
returns. Let $z_1$ be the first point where $\mathcal{J}=c^\dagger$
and $z_2$ the last; on the sub-segments $[x,z_1]$ and $[z_2,y]$ the
cost stays at most $c^\dagger$, so
$\lVert\nabla\mathcal{J}\rVert\le G(c^\dagger)$ there, and
integrating along each gives
$c^\dagger-c\le G(c^\dagger)\lVert z_1-x\rVert$ and
$c^\dagger-c\le G(c^\dagger)\lVert y-z_2\rVert$. Adding the two
completes the proof.
\end{IEEEproof}
The bound is explicit once $c^\dagger$ is known; in this paper
$c^\dagger$ is computed only for the two-dimensional counterexample,
and certifying a lower bound on it without first locating the
minimizer is the second open problem recorded in the conclusion of
the letter.
\section{Extended Numerical Studies}\label{app:experiments}
All experiments share one protocol. Minimizers are located by
gradient descent with Armijo backtracking from pointwise Riccati warm
starts, run to gradient norm $10^{-10}$ or the stated iteration cap.
Sample points on a level set are generated by drawing directions
uniformly on the unit sphere under a fixed seed, walking outward
geometrically until the target level is bracketed, and bisecting onto
it; radial shells add samples at relative radii between $10^{-3}$ and
$10^{-1}$. Seeds are fixed per experiment: 11--13 for the identity
check, 21--24 for the ratio sweep, 51--53 for the ratio
maximization, and 41--44 for the convergence runs. All computations
are dense double precision, with Lyapunov equations solved by
vectorization. Code and data reproducing every numerical result in
the letter and in this document are available at
\url{https://github.com/Rainlabuw/hidden-star-convexity}.
Table~\ref{tab:identity} refines the identity check of
\S\ref{sec:numerics} with median errors alongside the maxima; the
medians sit near the rounding unit of double precision. For the ratio
sweep, the complete record is Table~\ref{tab:tau} in the main body;
this appendix adds nothing to it.
\begin{table}[!ht]
\caption{Relative discrepancy between the two sides
of \eqref{eq:star}.}
\label{tab:identity}
\centering
\begin{tabular}{lccc}
\toprule
system & samples & max & median\\
\midrule
second-order & 240 & $1.3\times10^{-14}$ & $2.3\times10^{-16}$\\
Duffing & 240 & $2.7\times10^{-14}$ & $5.5\times10^{-16}$\\
wing & 144 & $5.1\times10^{-13}$ & $9.8\times10^{-16}$\\
\bottomrule
\end{tabular}
\end{table}
The ratio maximization behind \S\ref{sec:numerics} proceeds per
level: starts are drawn on the level set by ray bisection (even
starts) and at a uniform interior radius (odd starts); each start
ascends a forward-difference gradient of $\tau$ with backtracking,
and a step that leaves $\{\mathcal{J}\le c\}$ is projected back to
the boundary by bisection along its own direction, with eighty steps
per start on the second-order and Duffing systems and fifty on the
wing. The maximizing gains sit in the interior of the component: the
level-set maxima of Table~\ref{tab:tau} are lower because the
suprema are not attained on level sets. Ten starts of one hundred
twenty steps per level give wing maxima $0.98$, $1.02$, $1.02$,
$1.09$, $1.52$, and $1.24$ at
$c/\mathcal{J}^\star\in\{1.1,\,1.5,\,2,\,3,\,5,\,8\}$. The
violating gain of \S\ref{sec:numerics}, of cost
$1.37\,\mathcal{J}^\star$ with $\tau=1.011$, agrees across two
independent Lyapunov solvers to six digits, and the straight segment
from $\mathbf{K}^\star$ to it stays below cost
$1.370\,\mathcal{J}^\star$, so it lies in $\mathcal{G}_c$ at its
own level.
The constant steps of Fig.~\ref{fig:convergence} are chosen by a
probe rule: the largest $\eta=2^{-k}$ for which two hundred probe
iterations never increase the cost. The wing's Hessian at the located
minimizer, computed by central differences of the gradient, has
extreme eigenvalues $1.68\times10^{4}$ and $6.01\times10^{6}$, a
spread of roughly 360; entering through the product $\eta\mu$, this
spread is the quantitative source of the shallow wing slope in
Fig.~\ref{fig:convergence}.
The conservatism that motivates direct optimization is quantified on
the wing. A common-Lyapunov guaranteed-cost PDC design at the four
TP vertices (all vertex pairs, cost bound minimized by an
interior-point solver, feasibility of the returned solution verified
to $10^{-7}$) yields a gain with scheduled cost
$\mathcal{J}=6.69\times10^{4}$, against the policy-gradient optimum
$4.27\times10^{4}$: a factor of $1.57$. The pointwise-Riccati
least-squares fit used to initialize descent already attains the
optimal cost to four digits, so on this model the conservatism lies
in the common-Lyapunov synthesis, not in the difficulty of the
optimization.
\section{Detailed Analysis of the Counterexample}\label{app:trap}
The example of \eqref{eq:mirror} is built by symmetry breaking. At
$\nu=(\tfrac12,\tfrac12)$ the two plants exchange under
$K\mapsto-K$, so the landscape is symmetric and its two minima have
equal cost; descent has no reason to prefer one. Tilting the weights
to $\nu=(0.51,\,0.49)$ breaks the tie without destroying either
basin: one minimum becomes strictly global, the other strictly
spurious, and both persist because the basins are separated by a
ridge that a small tilt cannot remove. This is the smallest mechanism
we know that defeats descent for a scheduled cost, and it needs only
two plants and a single $1\times2$ gain.
Gradient descent from $K^0=[-0.6577,\ 0.1363]$ converges to
$\mathbf{K}^\star=[-0.6582,\ 0.1362]$ with
$\mathcal{J}^\star=113.514$; from $-K^0$ it converges to the
spurious minimizer $[0.6571,\ -0.1365]$ at cost $116.831$, both to
gradient norm $2\times10^{-5}$. The critical level was computed two
ways. Bisection on connectivity (the sublevel set
$\{\mathcal{J}\le\ell\}$ is thresholded on a grid over the gain
plane, its connected components containing the two minimizers are
tracked, and levels that join and levels that separate them bracket
$c^\dagger$) gives $c^\dagger\approx153.48$ on a $241\times241$
grid and $153.18$ on a $481\times361$ grid; pixel connectivity
requires a corridor one cell wide, so these estimates approach the
critical level from above. Newton's method on the gradient, started
inside the joining corridor, then locates the ridge saddle at
$K=[0.0158,\ -0.0038]$, where the gradient norm is below
$10^{-12}$ and the Hessian eigenvalues are $-1.8\times10^{2}$ and
$5.4\times10^{4}$; its cost is the exact critical level,
$c^\dagger=153.16$, quoted throughout this appendix. The two basins
remain separated over a cost range of nearly forty units above the
spurious value.
At the spurious point, \eqref{eq:transfer} forces
$\tau=1+(\mathcal{J}-\mathcal{J}^\star)/Q_H$; with the measured
gap $3.318$ and $Q_H=96.14$, the forced value is $1.0345106$, while
the ratio computed directly from the two gradient fields is
$1.0345103$, an agreement of $2.6\times10^{-7}$ in relative terms,
limited by the residual gradient at the located point. This is the
sharpest numerical corroboration of Corollary~\ref{cor:transfer} in
the paper: the two sides of the comparison share no computational
path.
The dependence of the certificate on the reference is concrete on
this example. With the reference taken at the spurious minimizer,
the construction is unchanged: over 500 samples of the ball of
radius $0.01$ around that point, the largest sampled ratio is
$-9.84$, and the guarantees of \S\ref{sec:rate}, applied with this
reference, certify linear convergence to the spurious point on its
basin. Evaluated with the same reference, the global minimizer
gives $\tau=0.965$: a stationary point below the reference value
satisfies $\tau<1$, since \eqref{eq:transfer} makes $\tau-1$
proportional to the cost difference, which is negative in that
direction. Detection of spurious stationary points is therefore a
consequence of taking the global minimizer as the reference, not a
property of the ratio itself.
One feature of Fig.~\ref{fig:mirror} deserves emphasis. The region
$\tau>1$ does not stop at the spurious basin: near the saddle it
intrudes into the component of $\mathbf{K}^\star$ well below
$c^\dagger$. On the $481\times361$ grid, the supremum of $\tau$
over $\mathcal{G}_c$ first reaches one near $c=132.2$; minimizing
$\mathcal{J}$ along the curve $\tau=1$ inside the component
sharpens this to $c_\tau=132.05$. With $\mathcal{J}^\star=113.51$
and $c^\dagger=153.16$, the hypothesis of Assumption~\ref{ass:tau}
fails at $47\%$ of the level range between the optimal value and
the critical level. The growth of the grid supremum is steady
rather than abrupt, passing through $-2.51$, $0.30$, $1.93$,
$3.01$, and $3.86$ at $20\%$, $40\%$, $60\%$, $80\%$, and $99\%$
of that range. Assumption~\ref{ass:tau}
therefore holds on $\mathcal{G}_c$ exactly for $c<c_\tau$. The
hypothesis fails well before the geometry does, conservatively,
which is the correct side to fail on; in particular, a component
free of spurious stationary points need not satisfy
Assumption~\ref{ass:tau}.
% ---------------- bibliography ----------------
% Same entries as the letter (embedded); extend as needed.


\begin{thebibliography}{00}
\bibitem{letter2026} S. Shakeri, P. Baranyi, and M. Mesbahi, ``Hidden
star-convexity in policy optimization for gain-scheduled LQR,''
submitted to \emph{IEEE Control Systems Letters}, 2026.
\bibitem{fazel2018global} M. Fazel, R. Ge, S. M. Kakade, and M. Mesbahi,
``Global convergence of policy gradient methods for the linear quadratic
regulator,'' in \emph{Proc. 35th Int. Conf. Mach. Learn.}, 2018,
pp. 1467--1476.
\bibitem{bu2019lqr} J. Bu, A. Mesbahi, M. Fazel, and M. Mesbahi,
``LQR through the lens of first order methods: Discrete-time case,''
2019, arXiv:1907.08921.
\bibitem{talebi2024submanifolds} S. Talebi and M. Mesbahi, ``Policy
optimization over submanifolds for linearly constrained feedback
synthesis,'' \emph{IEEE Trans. Autom. Control}, vol. 69, no. 5,
pp. 3024--3039, 2024.
\bibitem{shakeri2026rhpg} S. Shakeri, P. Baranyi, and M. Mesbahi,
``Receding-horizon policy gradient for polytopic controller
synthesis,'' 2026, arXiv:2603.29283.
\bibitem{watanabe2025hidden} Y. Watanabe and Y. Zheng, ``Revisiting
strong duality, hidden convexity, and gradient dominance in the
linear quadratic regulator,'' 2025, arXiv:2503.10964.
\bibitem{pctlqr2026} L. Cui and R. D. Braatz, ``LQR for systems with
probabilistic parametric uncertainties: A gradient method,'' 2026,
arXiv:2603.26080.
\bibitem{nesterov2006cubic} Y. Nesterov and B. T. Polyak, ``Cubic
regularization of Newton method and its global performance,''
\emph{Math. Program.}, vol. 108, no. 1, pp. 177--205, 2006.
\bibitem{necoara2019linear} I. Necoara, Y. Nesterov, and F. Glineur,
``Linear convergence of first order methods for non-strongly convex
optimization,'' \emph{Math. Program.}, vol. 175, no. 1--2,
pp. 69--107, 2019.
\bibitem{horn1991topics} R. A. Horn and C. R. Johnson, \emph{Topics
in Matrix Analysis}. Cambridge, U.K.: Cambridge Univ. Press, 1991.
\bibitem{ambrosetti1973dual} A. Ambrosetti and P. H. Rabinowitz,
``Dual variational methods in critical point theory and
applications,'' \emph{J. Funct. Anal.}, vol. 14, no. 4,
pp. 349--381, 1973.
\bibitem{hinder2020near} O. Hinder, A. Sidford, and N. Sohoni,
``Near-optimal methods for minimizing star-convex functions and
beyond,'' in \emph{Proc. 33rd Conf. Learn. Theory}, 2020,
pp. 1894--1937.
\bibitem{karimi2016linear} H. Karimi, J. Nutini, and M. Schmidt,
``Linear convergence of gradient and proximal-gradient methods under
the Polyak--{\L}ojasiewicz condition,'' in \emph{Proc. Joint Eur.
Conf. Mach. Learn. Knowl. Discov. Databases}, 2016, pp. 795--811.
\bibitem{takarics2014wing} P. Baranyi and B. Takarics, ``Aeroelastic wing
section control via relaxed tensor product model transformation
framework,'' \emph{J. Guid. Control Dyn.}, vol. 37, no. 5,
pp. 1671--1677, 2014.
\bibitem{rugh2000survey} W. J. Rugh and J. S. Shamma, ``Research on gain
scheduling,'' \emph{Automatica}, vol. 36, no. 10, pp. 1401--1425, 2000.
\bibitem{baranyi2004tp} P. Baranyi, ``TP model transformation as a way to
LMI-based controller design,'' \emph{IEEE Trans. Ind. Electron.}, vol. 51,
no. 2, pp. 387--400, 2004.
\bibitem{tanaka1998chaos} K. Tanaka, T. Ikeda, and H. O. Wang, ``A
unified approach to controlling chaos via an LMI-based fuzzy control
system design,'' \emph{IEEE Trans. Circuits Syst. I, Fundam. Theory
Appl.}, vol. 45, no. 10, pp. 1021--1040, 1998.
\bibitem{tanaka2001pdc} K. Tanaka and H. O. Wang, \emph{Fuzzy Control
Systems Design and Analysis: A Linear Matrix Inequality Approach}.
New York, NY, USA: Wiley, 2001.
\bibitem{drlqr2025} T. Fujinami, B. D. Lee, N. Matni, and G. J. Pappas,
``Policy gradient for LQR with domain randomization,'' 2025,
arXiv:2503.24371.
\bibitem{mamllqr2024} L. F. Toso, D. Zhan, J. Anderson, and H. Wang,
``Meta-learning linear quadratic regulators: A policy gradient MAML
approach for model-free LQR,'' in \emph{Proc. 6th Annu. Learn. Dyn.
Control Conf.}, ser. PMLR, vol. 242, 2024, pp. 902--915.
\bibitem{soflandscape2023} J. Duan, J. Li, X. Chen, K. Zhao, S. E. Li, and
L. Zhao, ``Optimization landscape of policy gradient methods for
discrete-time static output feedback,'' \emph{IEEE Trans. Cybern.},
vol. 54, no. 6, pp. 3588--3601, 2024.
\bibitem{giessler2025lpv} A. Gie{\ss}ler, F. Strehle, J. Illerhaus, and
S. Hohmann, ``Dynamic state-feedback control for LPV systems: Ensuring
stability and LQR performance,'' 2025, arXiv:2505.22248.
\end{thebibliography}
    \end{document}